\documentclass[aop,MSNbibl,seceqn,dvips]{arximspdf}
\usepackage{graphicx}
\doi{10.1214/13-AOP854} 
\volume{42}
\issue{3}
\pubyear{2014}
\firstpage{1197}
\lastpage{1211}

\makeatletter
\newcommand{\rrvert}{\vert}
\newcommand{\llvert}{\vert}
\newtheorem{theorem}{Theorem}[section]
\newtheorem{question}{Question}[section]
\newtheorem{lem}[theorem]{Lemma}
\newtheorem{prop}[theorem]{Proposition}
\newproclaim{definition}[theorem]{Definition}
\newproclaim{remark}{Remark}
\makeatother

\begin{document}
\begin{frontmatter}

\title{A simplified proof of the relation between scaling exponents in
first-passage percolation}
\pdftitle{A simplified proof of the relation between scaling exponents in
first-passage percolation}

\runtitle{Scaling relation in first-passage percolation}

\begin{aug}
\author[A]{\fnms{Antonio} \snm{Auffinger}\ead[label=e1]{auffing@math.uchicago.edu}}
\and
\author[B]{\fnms{Michael} \snm{Damron}\corref{}\ead[label=e2]{mdamron@math.princeton.edu}\thanksref{t1}}
\runauthor{A. Auffinger and M. Damron}
\pdfauthor{Antonio Auffinger and Michael Damron}
\affiliation{University of Chicago and Princeton University}
\address[A]{Department of Mathematics\\
University of Chicago\\
5734 S. University Avenue\\
Chicago, Illinois 60637\\
USA\\
\printead{e1}} 
\address[B]{
Department of Mathematics\\
Princeton University\\
Fine Hall, Washington Road\\
Princeton, New Jersey 08544\\
USA\\
\printead{e2}}
\end{aug}
\thankstext{t1}{Supported by an NSF Postdoctoral Fellowship.}

\received{\smonth{4} \syear{2012}}

%
\begin{abstract}
In a recent breakthrough work, Chatterjee [\textit{Ann. of Math. (2)}
\textbf{177} (2013) \mbox{663--697}] proved a long standing conjecture
that relates the transversal exponent $\xi$ and the fluctuation
exponent $\chi$ in first-passage percolation on $\mathbb{Z}^d$. The
purpose of this paper is to replace the main argument of Chatterjee
(2013) and give an alternative proof of this relation. Specifically, we
show that under the assumption that exponents defined in Chatterjee
(2013) exist, one has the relation $\chi\leq2 \xi-1$. One advantage of
our argument is that it does not require the ``nearly Gamma''
assumption of Chatterjee (2013).
\end{abstract}

%
\begin{keyword}[class=AMS]
\kwd{60K35}
\kwd{82B43}
\end{keyword}
\begin{keyword}
\kwd{First-passage percolation}
\kwd{KPZ relation}
\end{keyword}

\end{frontmatter}

\section{Introduction}\label{sec1}

We consider first-passage percolation (FPP) on $\mathbb{Z}^d$ with
nonnegative i.i.d. weights $(\tau_e)$ on edges with common distribution
$\mu$. For a review and a description of known results on the model we
refer the reader to \mbox{\cite{Blair-Stahn,Howard,Kestensurvey}}.

The random variable $\tau_e$ is called the \textit{passage time} of the
edge $e$, a nearest-neighbor edge in $\mathbb{Z}^d$. A \textit{path}
$\gamma$ is a sequence of edges $e_1, e_2,\ldots$ in $\mathbb{Z}^d$
such that for each $n \geq1$, $e_n$ and $e_{n+1}$ share exactly one
endpoint. For any finite path $\gamma$ we define the \textit{passage
time} of $\gamma$ to be $\tau(\gamma)=\sum_{e \in\gamma} \tau _e$, and
given two points $x,y \in\mathbb{R}^d$ we set
\[
\tau(x,y) = \inf_{\gamma} \tau(\gamma).
\]
The infimum is over all paths $\gamma$ that contain both $x'$ and $y'$,
and $x'$ is the unique vertex in $\mathbb{Z}^d$ such that $x \in x' +
[0,1)^d$ (similarly for $y'$). A minimizing path for $\tau(x,y)$ is
called a \textit{geodesic} from $x$ to $y$. We assume throughout the
paper that $\mu$ has no mass larger than or equal to $p_c(d)$, the
critical probability of bond percolation, at the infimum of its
support.

Consider a geodesic from the origin to a point $v$ with passage time
$\tau(0,v)$. One of the central questions \cite{Howard,Kestensurvey} in
this model (and in related ones) is to prove the following statement.
There exists an intrinsic relation between the magnitude of deviation
of $\tau(0,v)$ from its mean and the magnitude of deviation of the
geodesic $\tau(0,v)$ from a straight line joining $0$ and $v$. This
relation is \textit{universal}; that is, it is independent of the
dimension $d$ and of the law of the weights (as long they satisfy
certain moment assumptions).

The fluctuations of the passage time $\tau(0,v)$ about $\mathbb
{E}\tau(0,v)$ should be of order~$|v|^\chi$, where $\chi$ is called the
\textit{fluctuation exponent}. Analogously, a~\textit{transversal
exponent} $\xi$ should measure the maximal Euclidean distance of a
geodesic from $0$ to $v$ from the straight line that joins $0$ to $v$.
The intrinsic relation described above should be given as
%
%
\begin{equation}
\label{KPZ} \chi= 2 \xi-1.
\end{equation}
Despite numerous citations (both in mathematics and physics papers
\cite{KZ,Krug,LNP,NewmanPiza,Wuthrich}) and the mystery surrounding
(\ref{KPZ}), the existence and the ``correct'' definition of these
exponents is still not established, and these issues form part of the
above conjecture.

For a certain definition of the exponents, the inequality $ \chi\geq2
\xi-1 $ was proved and understood in the 1995 work of Newman and Piza
\cite{NewmanPiza}. The other inequality, however, has remained elusive
for more than twenty years. A~recent work of \mbox{Chatterjee
\cite{Sourav}} proposed a stronger definition of the exponents that
allows a complete proof of (\ref{KPZ}). One of its main contributions
was to give proof of the inequality $ \chi\leq2 \xi-1 $. The proof
relies on a construction similar to that in \cite{SouravPartha}. One
first breaks a geodesic into smaller segments and then uses an
approximation scheme to compare the passage time to a sum of nearly
i.i.d. random variables. The proof is then a trade-off between
minimizing the error while maximizing the variance of the passage time.
Assuming that the distribution is ``nearly Gamma'' (see
\cite{BenaimRossignol} for a~definition), the optimization can be
achieved by choosing different parameters in the approximation.

The main goal of this paper is to show a simple idea that enables us to
replace the main argument of \cite{Sourav} to prove the inequality $
\chi\leq2 \xi-1 $ in the case $\chi>0$. Our proof does not use a
``nearly Gamma'' assumption on the passage times, and so it applies to
all distributions for which Chatterjee's exponents exist. Furthermore,
this idea allows us to extend our theorem to related models like
directed polymers in random environments and last passage percolation.
These questions will be addressed in a forthcoming paper \cite{AD13}. A
secondary goal of this paper is to explain how the simplicity of our
proof could allow one to extract weaker assumptions on the model to
guarantee that different versions of (\ref{KPZ}) hold. It is important
to note that the cylinder construction we use has appeared in both
\cite{Johansson2} and \cite{Wuthrich} to prove two-dimensional versions
of the scaling relation for related models in either exactly solvable
or Poissonized cases.

We close this section by discussing earlier works related to
(\ref{KPZ}) and sketching the idea of Newman and Piza \cite{NewmanPiza}.
In 1993, Kesten \cite{Kesten} showed that $\chi\leq\frac{1}{2}$. In
1996, Licea, Newman and Piza \cite{LNP} proved\vspace*{-1pt} that
$\xi\geq\frac{1}{2}$ in all dimensions for one definition of $\xi$ and
$\xi\geq \frac{3}{5}$ in two dimensions for another definition. In
Section~\ref{cylinder}, we use their cylinder construction \cite{LNP}
as a fundamental tool to obtain the proof of~(\ref{KPZ}). The most
well-known conjecture after (\ref{KPZ}) is, however, that in two
dimensions one should have the exact values $\xi= \frac{2}{3}$ and
$\chi= \frac{1}{3}$.

The proof of Newman and Piza \cite{NewmanPiza} was based on an argument of
Aizemann and Wehr \cite{AW} (in a different context) with important
contributions from Alexander \cite{Alexander} and Kesten \cite{Kesten}.
Their main tool was an assumption of curvature of the limit shape
$\mathcal{B}_\mu$ [defined in (\ref{limitshape})] and the following
argument; see Figure~\ref{fig1}. Let $v$ be a unit vector. Assume that
a geodesic from $0$ to $nv$ leaves a box of height $n^\xi$ centered on
the straight line that joins the origin to $nv$ through a point $w$.
Furthermore, assume that the limit shape has shape curvature $2$ in the
direction of $v$; that is, there exists a positive constant $c$ such
that
%
%
\begin{equation}
\label{eqcurvature2} g(v):= \lim_{n\rightarrow\infty} \frac{\mathbb{E}(\tau
(0,nv))}{n}\quad\mbox{satisfies}\quad c|z|^2 \leq\bigl|g(v+z)-g(v)\bigr|
\end{equation}
for all vectors $z$ orthogonal to $v$ of small length. (A precise
definition of the shape curvature will be given in Section~\ref{comments}.) The passage time being additive in a geodesic implies that
%
%
\begin{equation}
\tau(0,nv) = \tau(0,w) + \tau(w,nv).
\end{equation}
Alexander's subadditive approximation theorem \cite{Alexander} [see
(\ref{lemalexander})] guarantees that $\tau(0,w) + \tau(w,nv) -
\mathbb{E}\tau(0,nv)$ is within $O(n^{\chi})$ of $g(w) + g(nv-w)
-g(nv)$, which is equal to $g(\lambda v -(\lambda v-w)) + g((n-\lambda
) v -(w-\lambda v)) -g(nv)$; see Figure~\ref{fig1}. By the curvature
assumption (\ref{eqcurvature2}) and by linearity of $g$ in the
direction of~$v$ this term is of order at least $\frac{|w-\lambda
v|^2}{n} = n^{2\xi-1}$ regardless of the choice of $w$. This
contradicts the fact that $\tau(0,nv) - \mathbb{E}\tau(0,nv)$ has order
$n^\chi$ if $\chi< 2\xi-1$, proving the lower bound.

%
\begin{figure}

\includegraphics{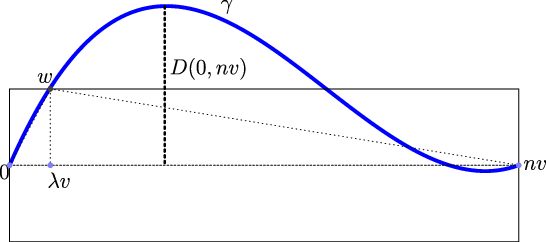}

\caption{A geodesic $\gamma$ from $0$ to $nv$ leaving a box of height
$n^{\xi}$ at a point $w$.}\label{fig1}
\end{figure}

\subsection{Outline of the paper}

In the next section, we state our main result and give a sketch of the
proof. Next, in Section~\ref{comments}, we discuss how our proof could
give rise to important extensions. In Section~\ref{proof} we prove
Theorem \ref{theorem1}. 

\section{Results}

Let $D(0,v)$ be the maximum Euclidean distance between the set of all
geodesics from $0$ to $v$ and the line segment joining $0$ to $v$. We
say that the FPP model has \textit{global exponents in the sense of
Chatterjee} if there exist real numbers $\chi_a$, $\chi_b$ and $\xi_a$,
$\xi_b$ such that:
\begin{enumerate}[(2)]
\item[(1)] For each choice of $\chi' > \chi_a$ and $\xi' > \xi_a$,
    there exists $\alpha>0$ so that
%
%
\begin{eqnarray}
\label{above} \sup_{v \in\mathbb{Z}^d \setminus\{0\}} \mathbb{E}\exp \biggl( \alpha
\frac{|\tau(0,v)-\mathbb{E}\tau(0,v)|}{|v|^{\chi'}} \biggr) &<& \infty\quad\mbox{and}
\nonumber\\[-8pt]\\[-8pt]
\sup_{v \in\mathbb{Z}^d \setminus\{0\}} \mathbb{E}\exp \biggl( \alpha\frac{D(0,v)}{|v|^{\xi'}}
\biggr) &<& \infty.\nonumber
\end{eqnarray}

\item[(2)] For each choice of $\chi''< \chi_b$ and all $\xi''< \xi_b$,
%
%
\begin{equation}
\label{below} \inf_{v \in\mathbb{Z}^d \setminus\{0\}} \frac{\operatorname{Var}(\tau
(0,v))}{|v|^{2\chi''}} >0\quad\mbox{and}\quad \inf
_{v \in\mathbb{Z}^d
\setminus\{0\}} \frac{\mathbb{E}(D(0,v))}{|v|^{\xi''}} >0.
\end{equation}
\end{enumerate}

%
\begin{remark}\label{remtrivial} It is not difficult to prove (see
\cite{Sourav}) that if such exponents exist, then $0\leq\xi_b \leq
\xi_a \leq1$ and $0\leq\chi_b \leq\chi_a \leq\frac{1}{2}$.
\end{remark}

Our main result is the following.

%
\begin{theorem}\label{theorem1} Assume that the FPP model has global
exponents in the sense of Chatterjee and $\chi:= \chi_a = \chi_b
>0$. Then
%
%
\begin{equation}
\label{KPZu} \chi\leq2 \xi_a -1.
\end{equation}
\end{theorem}

%
\begin{remark}
Our proof does not require one to assume that the distribution $\mu$ of
the $\tau_e's$ is \textit{nearly Gamma} as in \cite{Sourav}. The case
$\chi=0$ was treated with a separate argument in \cite{Sourav}. It does
not require this assumption on $\mu$, and although it is stated for
continuous distributions only, the arguments hold under our condition
on the support of $\mu$.
\end{remark}

In \cite{Sourav} it was shown using the ideas of Newman and Piza
\cite{NewmanPiza} and Howard \cite{Howard} that for this definition of
exponents, the lower bound holds,
%
%
\begin{equation}
\chi_a \geq2 \xi_b -1.
\end{equation}
This fact, combined with Theorem \ref{theorem1} and with the assumption
$\xi_a = \xi_b$, implies:

%
\begin{theorem} Assume that the FPP model has global exponents in the
sense of Chatterjee with $\chi:= \chi_a = \chi_b$ and $\xi:= \xi_a =
\xi_b$. Then (\ref{KPZ}) holds.
\end{theorem}

\subsection{Sketch of the proof}
In this subsection we sketch the proof of Theorem~\ref{theorem1}. It
will follow from the picture below.
%
\begin{figure}

\includegraphics{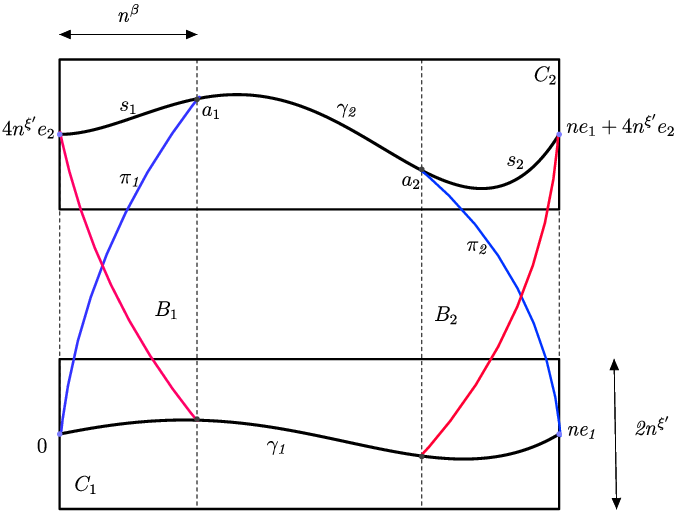}

\caption{$B_1$ and $B_2$ are boxes of length $n^{\beta}$ and radius
$3n^{\xi'}$ and $\pi_i, i=1,2$ are geodesics joining the points $0$
to $a_1$ and $a_2$ to $ne_1$, respectively. Since $\gamma_1$ is a
geodesic, the path from $0$ to $ne_1$ using $\pi_1$, $\pi_2$ and the
middle part of $\gamma_2$ has larger passage time than $\gamma_1$.
Note $\beta$ is chosen larger than $\xi'$ so the boxes $B_1$ and
$B_2$ are not drawn to scale.}\label{figMichael}
\end{figure}

Look at the cylinders $C_1$ and $C_2$ in Figure~\ref{figMichael}. They
are both of length $n$ and radius $n^{\xi'}$ for some $\xi'>\xi_a$. The
top cylinder is identical to the bottom but shifted up (in direction
$e_2$, the second coordinate vector) by $4n^{\xi'}$. The dark paths
$\gamma_1$~and~$\gamma_2$ joining $0$ to $ne_1$ and their shifted
points are geodesics.

Since we chose $\xi' > \xi_a$ it is possible to show that the passage
times $\tau(\gamma_1)$ and $\tau(\gamma)$ are almost independent. Using
(\ref{below}), this implies that for any $\chi''<\chi$ and $n$ large,
%
%
\begin{equation}
\label{eqslb} n^{2\chi''}\leq\operatorname{Var}\bigl(\tau(
\gamma_2) - \tau(\gamma_1)\bigr).
\end{equation}

Assuming $\xi_a<1$ and $\xi' < \beta<1$, build two cylinders, $B_1$ and
$B_2$, of length $n^\beta$ and radius $3n^{\xi'}$ as in the picture.
Let $a_1$ and $a_2$ be the last and first points of intersection of the
geodesic $\gamma_2$ with these cylinders. Consider the geodesics
$\pi_1$ joining $0$ to $a_1$ and $\pi_2$ joining $a_2$ to $ne_1$ (in
blue). Note that the concatenation of $\pi_1$, the piece of $\gamma_2$
from $a_1$ to $a_2$ and $\pi_2$ is a path from $0$ to $ne_1$.
Therefore, if $s_1$ and $s_2$ are the other two parts of $\gamma_2$ (as
in the picture),
\[
\tau(\gamma_1) \leq\tau(\pi_1) + \tau(a_1,a_2)
+\tau(\pi_2) = \tau(\pi_1) + \bigl(\tau(
\gamma_2) - \tau(s_1) - \tau(s_2) \bigr) +
\tau(\pi_2),
\]
which implies
%
%
\begin{equation}
\label{se1} \tau(\gamma_1) - \tau(\gamma_2) \leq\tau(
\pi_1) - \tau(s_1) + \tau(\pi_2) -
\tau(s_2).
\end{equation}

The difference $\tau(\pi_1) - \tau(s_1)$ is bounded above by
\[
X:= \max_{u,v,u',v' \in\partial B_1 } \tau(u,v) - \tau\bigl(u',v'
\bigr),
\]
where $u$ and $u'$ are points on the left boundary of the box $B_1$
while $v$ and $v'$ are points on the right boundary of the box. Using
the box $B_2$ one can similarly bound the difference of $\tau(\pi_2) -
\tau(s_2)$ by a random variable with same distribution as $X$. Using
the red paths instead of the blue ones and reversing the roles of
$\gamma_1$ and $\gamma_2$ in (\ref{se1}), we get an inequality for the
absolute value of the left-hand side of (\ref{se1}). Combining these
bounds,
\[
\operatorname{Var}\bigl(\tau(\gamma_2) - \tau(\gamma_1)
\bigr) \leq4 \mathbb {E}X^2.
\]

For $\mathbb{E}X^2$ it suffices to bound (independently of $u$ and
$v$) the second moment of
%
%
\begin{eqnarray}
&& \bigl|\tau(u,v) - \tau\bigl(0,n^\beta e_1\bigr)\bigr|\nonumber
\\
&&\qquad \leq \bigl|\tau(u,v)-\mathbb{E}\tau (u,v)\bigr| + \bigl| \tau\bigl(0,n^\beta e_1\bigr) -
\mathbb{E}\tau\bigl(0,n^\beta e_1\bigr)\bigr|
\nonumber\\[-8pt]\\[-8pt]
&&\quad\qquad{}+ \bigl|g(v-u) -\mathbb{E}\tau(u,v)\bigr| + \bigl| g(ne_1) - \mathbb{E}\tau
\bigl(0,n^\beta e_1\bigr)\bigr|\nonumber
\\
&&\quad\qquad{}+ \bigl|g(v-u)-g\bigl(n^\beta e_1\bigr)\bigr|.\nonumber
\end{eqnarray}
The first two lines above are bounded above by $n^{\beta\chi'}$ for any
$\chi'>\chi$ (by assumption and Alexander's subadditive approximation)
while the third is of order $n^{2\xi' - \beta}$ by the curvature of the
limit shape $\mathcal{B}_\mu$; see (\ref{limitshape}). This implies by
(\ref{eqslb}) and the above computation,
%
%
\begin{equation}
n^{2\chi''} \leq\operatorname{Var}\bigl(\tau(\gamma_2) - \tau(
\gamma _1)\bigr) \leq C \bigl(n^{2\beta\chi'} + n^{2 (2\xi' - \beta)}
\bigr).
\end{equation}
Now choosing $\chi''$ and $\chi'$ close enough to $\chi$ and
recalling that $\beta<1$, we get $n^{2\chi''} \leq C n^{2 (2\xi' -
\beta)}$ for large $n$. This implies $\chi'' \leq2\xi' - \beta$.
Taking $\beta\uparrow1$, $\chi'' \uparrow\chi$ and $\xi'
\downarrow\xi$ ends the proof.

\section{Extensions}\label{comments}

In this section we discuss how to improve Theorem \ref{theorem1}. There
are two main directions. The first one is to establish a relation for
directionally defined exponents. This would weaken our assumptions,
allowing us to prove the existence of both exponents more easily. The
second is to add shape curvature into relation (\ref{KPZ}).

One can define the exponents $\xi_a, \chi_a, \chi_b$ directionally
as follows. For a unit vector $u$, define the cylinder $\mathcal C (u,
a, b)$ of length $a$ and radius $b$ in the direction $u$ as the set of
points in $\mathbb{R}^d$ at most $\ell_\infty$ distance $b$ away
from the line segment connecting 0 to $au$. We denote $\partial^f
\mathcal C(u, a, b)$ as the set of all points $x \in\mathcal C (u, a,
b) $ with $|\langle u,x \rangle| \geq a$.

The exponent $\xi_{a}^u$ is now defined as in (\ref{above}) with $v$
taken as a nonzero multiple of $u$ instead of an arbitrary vector in
$\mathbb{Z}^d \setminus\{0\}$. $\chi_a^u$ is defined similarly to
(\ref{above}) but as a function of $\xi_a^u$; it is the smallest real
number such that for any $\chi'>\chi_{a}^{u}$, there exists $\alpha$ so
that
\[
\inf_{\xi' >\xi_a^u} \sup_{n\in\mathbb{N}} \sup
_{v \in\mathcal
\partial^f C(u,n,n^{\xi'})} \mathbb{E}\exp \biggl( \alpha\frac
{|\tau(0,v)-\mathbb{E}\tau(0,v)|}{|v|^{\chi'}} \biggr) <
\infty.
\]
$\chi_b^u$ is defined as the largest real number such that for any
$\chi'' < \chi_b^u$,
\[
\inf_{n\in\mathbb{N}} \frac{\operatorname{Var}\tau
(0,nu)}{n^{2\chi''}} >0.
\]

One can go through the proof of Theorem \ref{theorem1} and see that the
scaling relation~(\ref{KPZ}) holds with these new exponents as long one
is able to prove that Alexander's subadditive exponent can be made
directional. Namely, the question becomes the following:

%
\begin{question}\label{question1} Is it true that for any $\chi'>\chi
_{a}^u$ there exists $\xi' >\xi_a^u$ and a~constant $C=C(\chi',\xi ')
>0$ such that for all $x \in\partial^f \mathcal C(u,n,n^{\xi'})$ and
all $n$,
\[
\bigl|\mathbb{E}\tau(0,x) - g(x)\bigr| \leq C|x|^{\chi'}.
\]
\end{question}

Another way to generalize the relation (\ref{KPZ}) is to add curvature.
Let
%
%
\begin{equation}
\label{limitshape} \mathcal{B}_\mu:= \bigl\{x \in\mathbb{R}^d,
g(x)\leq1\bigr\}
\end{equation}
be the limit shape of the model; see \cite{Kestensurvey}. Let $u$ be a
unit vector of $\mathbb{R}^d$ and let $H_0$ be a hyperplane such that
$u+H_0$ is tangent to $g(u) \mathcal{B}_\mu$ at $u$. We introduce a
third exponent, called the \textit{curvature exponent} as follows.

%
\begin{definition}\label{def1}
The curvature exponent $\kappa^u$ in the direction $u$ is a real
number such that there exist positive constants $c$, $C$ and
$\varepsilon$ such that for any $z \in H_0$ with $|z|<\varepsilon$,
one has
%
%
\begin{equation}
\label{eqcurvatureassumption} c|z|^{\kappa^u} \leq g(u+z)-g(u) \leq
C|z|^{\kappa^u}.
\end{equation}
\end{definition}

The directional approach mentioned above together with the definition
of the curvature exponent allows us to generalize relation (\ref{KPZ})
to one that includes all three of these exponents. Assume that
Question~\ref{question1} is answered affirmatively and that $\chi^u:=
\chi _a^u = \chi_b^u$ $(\geq0)$. Then it would follow directly from the
proof of Theorem~\ref{theorem1} that (\ref{KPZu}) generalizes to
%
%
\begin{equation}
\label{ADrelation} \chi^u \leq\kappa^u
\xi^u_a - \bigl(\kappa^u -1\bigr).
\end{equation}
Moreover, if $\xi^u:=\xi_a^{u} = \xi_b^{u}$, then one would have
%
%
\begin{equation}
\label{ADrelation2} \chi^u = \kappa^u \xi^u -
\bigl(\kappa^u -1\bigr).
\end{equation}

%
\begin{remark} Note that when $\kappa^u = 2$ and the exponents are
global, (\ref{ADrelation2}) is the same as (\ref{KPZ}). This is
believed to be true in the case where the weights $\tau$ have
a~continuous distribution with finite exponential moments. It would be
of interest to find examples, maybe of other growth models, where
(\ref{ADrelation}) holds for $\kappa^u \neq2$.
\end{remark}

%
\begin{remark} It is unclear if Chatterjee's exponents exist and, if
so, what the implications would be. For example, existence immediately
implies that $\kappa^u \leq2$ in all directions where $\kappa^u$ is
defined. In particular the limit shape can not contain flat pieces as
in \cite{DurrettLiggett}. However, if the statement in
Question~\ref{question1} holds and if one uses directional exponents
(provided they exist), then it would be possible to show that the upper
bound in (\ref{eqcurvatureassumption}) holds for all $\kappa$ (this is
true, for example, if there is a flat edge in direction $u$) if and
only if $\xi_a^u=1$.
\end{remark}

\section{\texorpdfstring{Proof of Theorem \protect\ref{theorem1}}
{Proof of Theorem 2.1}}\label{proof}
\subsection{Preliminary lemmas}
Recall the definition of the function $g$ from (\ref{eqcurvature2}). We
first state a bound on the ``nonrandom fluctuations'' from
\cite{Alexander}. The proof of the lemma, as stated, can be found in
\cite{Sourav}.

%
\begin{lem}\label{lemalexander}
For any $\chi'>\chi_a$, there exists $C_1 = C_1(\chi')>0$ such that for
all $x \in\mathbb{R}^d$,
\[
\bigl|\mathbb{E}\tau(0,x)- g(x)\bigr| \leq C_1|x|^{\chi'}.
\]
\end{lem}

For a unit vector $x_0$, let $H_0$ be as in Definition \ref{def1}
(taking $u=x_0)$. 
For $m,n \geq1$ and $i=1,2$, set
\[
S_i(x_0; m,n) = \bigl\{x \in(i-1)nx_0+
H_0\dvtx  \bigl|x-(i-1)nx_0\bigr| \leq m\bigr\}
\]
and
\[
X(x_0;m,n) = \mathop{\max_{v_1,v_2 \in S_1(x_0;m,n)}}_{w_1,w_2 \in
S_2(x_0;m,n)} \bigl|
\tau(v_1,w_1)-\tau(v_2,w_2)\bigr|.
\]

The following proposition is a slight modification of the arguments in
\cite{Sourav}. For a random variable $G$, write $\|G\|_2$ for the $L^2$
norm $(\mathbb{E}G^2)^{1/2}$.
%
\begin{prop}\label{propsourav}
Let $|x_0|=1$, and assume (\ref{eqcurvatureassumption}) holds for some
$C_{x_0}$, $\kappa$~and~$\varepsilon_{x_0}$. For each $\chi'>\chi _a$
there exists $C_2 = C_2(d,\chi')$ such that if $m,n$ have $m \leq
(\varepsilon_{x_0}/2\sqrt{d-1}) n$, then
\[
\bigl\|X(x_0;m,n)\bigr\|_2 \leq C_2n^{1-\kappa}
m^\kappa+ C_2n^{\chi'}.
\]
\end{prop}

\begin{pf}
By the triangle inequality, it suffices to bound the variable $Y$,
\[
Y(x_0;m,n) = \mathop{\max_{v \in S_1(x_0;m,n)}}_{w \in S_2(x_0;m,n)} \bigl|
\tau(v,w)-\tau(0,nx_0)\bigr|.
\]
For $v \in S_1(x_0;m,n)$ and $w \in S_2(x_0;m,n)$, the idea is to use
the following decomposition:
%
%
\begin{eqnarray}
&& \bigl|\tau(0,nx_0)-\tau(v,w)\bigr|\nonumber
\\
\label{eqsummand1} &&\qquad \leq\bigl|\tau(0,nx_0)-\mathbb{E}
\tau(0,nx_0)\bigr| + \bigl|\tau(v,w)-\mathbb {E}\tau(v,w)\bigr|
\\
\label{eqsummand2} &&\quad\qquad{} + \bigl|\mathbb{E} \tau(0,nx_0) -
g(nx_0)\bigr| + \bigl|\mathbb{E} \tau(v,w) - g(w-v)\bigr|
\\
\label{eqsummand3} &&{}\quad\qquad{} + \bigl|g(nx_0)-g(w-v)\bigr|.
\end{eqnarray}

We first estimate (\ref{eqsummand3}),
\[
\bigl|g(nx_0) - g(w-v)\bigr| = n \bigl|g(x_0)- g\bigl(x_0
+ (w-v)/n - x_0\bigr)\bigr|.
\]
By assumption, $|(w-v)/n - x_0| = (1/n) |w-v - nx_0| \leq2(m/n) \sqrt
{d-1} \leq\varepsilon_{x_0}$. Therefore, we can apply
(\ref{eqcurvatureassumption}) and find $C_3$ such that
%
%
\begin{equation}
\label{eqsummand3bound} \bigl|g(nx_0)-g(w-v)\bigr| \leq C_{x_0}n
\bigl|(w-v)/n - x_0\bigr|^\kappa\leq C_3 n^{1-\kappa}
m^\kappa.
\end{equation}

For (\ref{eqsummand2}), we note that $|w-v| \leq2n$ for all $w,v$. So
by Lemma~\ref{lemalexander},
%
%
\begin{equation}
\label{eqsummand2bound} \bigl|\mathbb{E} \tau(0,nx_0) -
g(nx_0)\bigr| + \bigl|\mathbb{E} \tau(v,w) - g(w-v)\bigr| \leq3C_1
n^{\chi'}.
\end{equation}
We turn to contributions to $Y(x_0;m,n)$ from terms in
(\ref{eqsummand1}). Pick $\hat{\chi} = (1/2)(\chi_a+\chi')$ and
\[
X:= \mathop{\max_{v \in S_1(x_0;m,n)}}_{w \in S_2(x_0;m,n)} \frac
{|\tau(v,w)-\mathbb{E} \tau(v,w)|}{|w-v|^{\hat{\chi}}}.
\]
By the fact that $\hat{\chi}>\chi_a$, for some $C_4$ and $C_5$,
\begin{eqnarray*}
\mathbb{E}e^{\alpha X} &\leq&\mathop{\sum_{v \in S_1(x_0;m,n)}}_{w
\in S_2(x_0;m,n)}
\mathbb{E} \biggl( \exp \biggl[\alpha\frac{|\tau
(v,w)-\mathbb{E}\tau(v,w)|}{|w-v|^{\hat{\chi}}} \biggr] \biggr)
\\
&\leq& C_4\bigl|S_1(x_0;m,n)\bigr|^2
\\
&\leq& C_5 m^{2(d-1)}.
\end{eqnarray*}
Since $\alpha>0$ and $X$ is positive, we may use Jensen's inequality
to get
%
%
\begin{eqnarray}
\label{eqpizza} e^{\alpha\|X\|_2} &=& 1 + \alpha\|X\|_2 + \sum
_{n=2}^\infty\frac
{(\alpha\|X\|_2)^n}{n!}\nonumber
\\
&\leq& 1 + \alpha\|X\|_2 + \mathbb{E}\sum_{n=2}^\infty \frac{(\alpha
X)^n}{n!}
\\
&\leq& \alpha\|X\|_2 + \mathbb{E}e^{\alpha X}. \nonumber
\end{eqnarray}
Because $e^{\alpha t} \geq2\alpha t$ for all $t \in\mathbb{R}$, it
cannot be that $\alpha\|X\|_2$ is the maximum of the two terms on the
right-hand side of (\ref{eqpizza}). Thus an upper bound is $2\mathbb
{E}e^{\alpha X}$, and taking logarithms of both sides, we find $\|X\|_2
\leq\frac{1}{\alpha} \log2\mathbb{E}e^{\alpha X}$. So $\|X\|_2 \leq C_6
\log m$ for some $C_6$. Let
\[
X':= \mathop{\max_{v \in S_1(x_0;m,n)}}_{w \in S_2(x_0;m,n)} \bigl|\tau
(v,w)-\mathbb{E}\tau(v,w)\bigr|.
\]
Since $|w-v| \leq2n$ for all $v \in S_1(x_0;m,n)$ and $w \in
S_2(x_0;m,n)$, $X' \leq C_7n^{\hat{\chi}}X$. Therefore, $\|X'\|_2 \leq
C_8n^{\chi'}$. We finish by putting this together with
(\ref{eqsummand3bound}) and~(\ref{eqsummand2bound}),
\[
\bigl\|Y(x_0;m,n)\bigr\|_2 \leq C_3
n^{1-\kappa}m^\kappa+ C_9n^{\chi'}.
\]\upqed
\end{pf}

To end the section, we give one general lemma about random variables.
Denote by $I(A)$ the indicator function of the event $A$.
%
\begin{lem}\label{lemrvs}
Let $X$ and $Y$ be random variables with $\|X\|_4,\|Y\|_4<\infty$, and
let $B$ be an event such that
\[
(X-Y)I(B) = 0\qquad\mbox{almost surely.}
\]
Then
%
%
\begin{equation}
\label{eqrvs} |\operatorname{Var}X -\operatorname{Var}Y| \leq\bigl(\|X
\|_4+\|Y\|_4\bigr)^2 \mathbb{P}
\bigl(B^c\bigr)^{1/4}.
\end{equation}
\end{lem}

\begin{pf}
Let $\widetilde X = X - \mathbb{E} X$ and $\widetilde Y = Y-\mathbb {E}
Y$. The left-hand side of (\ref{eqrvs}) equals
\begin{eqnarray*}
\bigl\llvert \|\widetilde X\|_2^2 - \|\widetilde Y
\|_2^2\bigr\rrvert &=& \bigl| \| \widetilde X\|_2 - \|\widetilde Y\|_2\bigr|
\bigl|\|\widetilde X\|_2 + \|\widetilde Y\|_2\bigr|
\\
&\leq& \|X-Y\|_2 \bigl(\|X\|_2 + \|Y\|_2\bigr)
\\
&\leq&\bigl\|(X-Y)I\bigl(B^c\bigr)\bigr\|_2 \bigl(\|X\|_4 +\|Y\|_4\bigr)
\\
&\leq& \|X-Y\|_4\bigl(\|X\|_4+\|Y\|_4\bigr) \mathbb{P}\bigl(B^c\bigr)^{1/4},
\end{eqnarray*}
which implies the lemma.
\end{pf}

\subsection{Cylinder construction}\label{cylinder}
Pick $x_0$ of unit norm and $H_0$ a hyperplane as in Definition
\ref{def1}. Fix an orthonormal basis $x_1,\ldots, x_{d-1}$ of $H_0$.
Let $T_1(x_0;n) = \tau(0, nx_0)$, $T_2(x_0;n,\xi') = \tau(4n^{\xi'}x_1,
nx_0+4n^{\xi'}x_1)$ and
\[
\delta T\bigl(x_0;n,\xi'\bigr) = T_1(x_0;n)
- T_2\bigl(x_0;n,\xi'\bigr).
\]
The idea will be to give a lower bound for the variance of $\delta T$
(Section~\ref{subseclowerbound}) and then an upper bound
(Section~\ref{subsecupperbound}). Comparing them, we obtain the desired
inequalities, (\ref{KPZu}) and (\ref{ADrelation}). This idea was
introduced by Licea, Newman and Piza in~\cite{LNP} and also used in
\cite{Johansson2} and \cite{Wuthrich}.

\subsubsection{\texorpdfstring{Lower bound on $\operatorname{Var}\delta T$}
{Lower bound on Var delta T}}\label{subseclowerbound}
We will now assume that
%
%
\begin{equation}
\label{eqassumption2} \xi_a<1\quad\mbox{and}\quad
\chi_b >0
\end{equation}
so that we can choose $\xi'$ and $\chi''$ such that
%
%
\begin{equation}
\label{eqxichiprime} \xi_a<\xi'<1\quad\mbox{and}\quad 0<
\chi''<\chi_b.
\end{equation}

%
\begin{prop}
Assume (\ref{eqassumption2}). For each $\xi'$ and $\chi''$ chosen as in
(\ref{eqxichiprime}), there exists $C=C(\xi',\chi'')$ such that for all
$n$,
\[
\operatorname{Var}\delta T\bigl(x_0; n,\xi'\bigr)
\geq Cn^{2\chi''}.
\]
\end{prop}

\begin{pf}
Define $\mathcal{C}_1$ as the set of points in $\mathbb{R}^d$ at most
$\ell_\infty$ distance $n^{\xi'}$ away from the line segment
connecting 0 to $nx_0$. Define $\mathcal{C}_2$ as the set of points at
most $\ell_\infty$ distance $n^{\xi'}$ away from the line segment
connecting $4n^{\xi'}x_1$ to $nx_0+4n^{\xi'}x_1$. Let $T_1(x_0;n)'$
and $T_2(x_0;n,\xi')'$ be as follows:
\begin{enumerate}[(2)]
\item[(1)] $T_1(x_0;n)'$ is the passage time from $0$ to $nx_0$ using
only edges with endpoints in $\mathcal{C}_1$.

\item[(2)] $T_2(x_0;n,\xi')'$ is the passage time from $4n^{\xi
'}x_1$ to $nx_0+4n^{\xi'}x_1$ using only edges with endpoints in
$\mathcal{C}_2$.
\end{enumerate}
Let $B$ be the event $\{T_1(x_0;n)=T_1(x_0;n)'$ and $T_2(x_0;n,\xi') =
T_2(x_0;n,\xi')'\}$. Note that if $T_1(x_0;n) \neq T_1(x_0;n)'$, then
$D(0,nx_0) \geq n^{\xi'}$. A similar statement holds for
$T_2(x_0;n,\xi')$ and $T_2(x_0;n, \xi')'$. Therefore, $\mathbb{P}(B^c)
\leq2 \mathbb{P}(D(0,nx_0) \geq n^{\xi'})$. Picking $\xi'' = (1/2)(\xi'
+ \xi_a)$, so that $\xi_a<\xi''<\xi'<1$, we find from the definition of
$\xi_a$ [from (\ref{above})] that there exists $C_1>0$ such that for
all $n$, $\mathbb{P}(D(0,nx_0) \geq n^{\xi'}) \leq
e^{-C_1n^{\xi'-\xi''}}$. Therefore
%
%
\begin{equation}
\label{eqprobbc} \mathbb{P}\bigl(B^c\bigr) \leq2e^{-C_1n^{\xi'-\xi''}}.
\end{equation}

By Lemma~\ref{lemrvs} with $X=\delta T(x_0;n,\xi')$ and $Y=\delta
T(x_0;n,\xi')':= T_1(x_0;n)' - T_2(x_0;n,\xi')'$:
\begin{eqnarray*}
&& \operatorname{Var}\delta T\bigl(x_0;n,\xi'\bigr) -
\operatorname{Var}\delta T\bigl(x_0;n,\xi'
\bigr)'
\\
&&\qquad \geq -\bigl(\bigl\|\delta T\bigl(x_0;n,\xi'\bigr)
\bigr\|_4 + \bigl\|\delta T\bigl(x_0;n,\xi'
\bigr)'\bigr\|_4\bigr)^2 \mathbb{P}\bigl(B^c\bigr)^{1/4}
\\
&&\qquad \geq- C_2n^2 e^{-C_1/4n^{\xi'-\xi''}}
\end{eqnarray*}
for some $C_2$. Here we have used inequality (\ref{eqprobbc}) and that
each $\delta T$ is a difference of two passage times, each of which has
$L^4$ norm bounded above by $C n$ (compare, e.g., to a deterministic
path). Therefore, there exists $C_3$ such that for all $n$,
%
%
\begin{equation}
\label{eqlast1} \operatorname{Var}\delta T\bigl(x_0;n,
\xi'\bigr) \geq\operatorname{Var}\delta T\bigl(x_0;n,
\xi'\bigr)' - C_3.
\end{equation}
But $\delta T(x_0;n,\xi')'$ is the difference of i.i.d. random
variables distributed as $T_1(x_0;n)'$, so
%
%
\begin{equation}
\label{eqlast2} \operatorname{Var}\delta T\bigl(x_0;n,
\xi'\bigr)' = 2 \operatorname{Var} T_1(x_0;n)'.
\end{equation}
By exactly the same argument as that given above, we can find $C_4$
such that for all~$n$,
\[
\operatorname{Var} T_1(x_0;n)' \geq
\operatorname{Var}T_1(x_0;n) - C_4 =
\operatorname{Var}\tau(0,nx_0) - C_4.
\]
Using the definition of $\chi''$, we can find another $C_5$ such that
for all $n$, $\operatorname{Var} \tau(0,nx_0) \geq
C_5n^{2\chi''}$. Combining this with (\ref{eqlast2}) and
(\ref{eqlast1}), we complete the proof.
\end{pf}

\subsubsection{\texorpdfstring{Upper bound on $\operatorname{Var}\delta T$}
{Upper bound on Var delta T}}\label{subsecupperbound}

In this section we continue to assume (\ref{eqassumption2}) and we work
with the same choice of $\xi'$ that satisfies (\ref{eqxichiprime}). We
will prove the following.

%
\begin{prop}
Assume (\ref{eqassumption2}) and that (\ref{eqcurvatureassumption})
holds for some $C,\varepsilon_{x_0}$ and $\kappa $. For each $\beta$
satisfying $\xi'<\beta<1$ and each $\chi'>\chi _a$, there exists
$C=C(\beta,\chi')$ such that for all $n$,
\[
\operatorname{Var}\delta T\bigl(x_0;n,\xi'\bigr) \leq
Cn^{2\beta(1-\kappa
)+2\xi'\kappa} + Cn^{2\beta\chi'}.
\]
\end{prop}

\begin{pf}
Define the hyperplanes
\[
H_1 = n^\beta x_0 + H_0\quad\mbox{and}\quad H_2 = \bigl(n-n^\beta\bigr)x_0 +
H_0.
\]
Let $\mathcal{C}_1$ and $\mathcal{C}_2$ be as in the proof of the lower
bound. For two points $a$ and $b$ in~$\mathbb{R}^d$, let $S(a,b)$ be
the set of finite paths $P$ from $a$ to $b$ (or their closest lattice
points) such that for both $i=1$ and $2$, $P \cap H_i \cap
[\mathcal{C}_1 \cup\mathcal{C}_2 ] \neq\varnothing $. Define
$T_1(x_0;n)''$, $T_2(x_0;n,\xi')''$ as follows:
\begin{enumerate}[(2)]
\item[(1)] $T_1(x_0;n)''$ is the minimum passage time of all paths in
$S(0,nx_0)$.

\item[(2)] $T_2(x_0;n,\xi')''$ is the minimum passage time of all
paths in $S(4n^{\xi'}x_1,nx_0+4n^{\xi'}x_1)$.
\end{enumerate}
Again we set $B$ equal to the event $\{T_1(x_0;n) = T_1(x_0;n)''$ and
$T_2(x_0;n,\xi') = T_2(x_0;n,\xi')''\}$. Because $B^c$ implies that
$D(0,nx_0) \geq n^{\xi'}$ [or the corresponding statement for
$T_2(x_0;n,\xi')$], we may choose $C_1$ such that for all $n$,
$\mathbb{P}(B^c) \leq2e^{-C_1n^{\xi'-\xi''}}$, where $\xi'' =
(1/2)(\xi'+\xi_a)$. Therefore, we can argue exactly as in the previous
section to find $C_2$ such that for all $n$,
%
%
\begin{equation}
\label{eqdecompose} \operatorname{Var}\delta T\bigl(x_0;n,
\xi'\bigr) \leq\operatorname{Var}\delta T\bigl(x_0;n,
\xi'\bigr)'' + C_2,
\end{equation}
where $\delta T(x_0;n,\xi')'' = T_1(x_0;n)''-T_2(x_0;n,\xi')''$.

For almost every passage time realization, we may define a path $\gamma
_1 \in S(0,nx_0)$ (in a measurable and deterministic way when there are
not unique geodesics) from~0 to $nx_0$ so that $\tau(\gamma_1) =
T_1(x_0;n)''$ and $\gamma_2 \in$ $S(4n^{\xi'}x_1, nx_0 + 4n^{\xi
'}x_1)$ such that $\tau(\gamma_2) = T_2(x_0;n,\xi')''$. Let $a_1$ be
the last lattice point on $\gamma_2$ before it intersects $H_1 \cap
(\mathcal{C}_1 \cup\mathcal{C}_2)$ and $a_2$ the last lattice point of
$\gamma_2$ before it intersects $H_2 \cap(\mathcal{C}_1 \cup
\mathcal{C}_2)$. Similarly let $a_1'$ be the last lattice point of
$\gamma_1$ before it intersects $H_1 \cap(\mathcal{C}_1 \cup
\mathcal{C}_2)$ and $a_2'$ the last lattice point of $\gamma_1$ before
it intersects $H_2 \cap(\mathcal{C}_1 \cup\mathcal{C}_2)$.
Write\vspace*{-1pt} $s_1$ for the piece of $\gamma_2$ (seen\vspace*{1pt} as an
oriented path) from $4n^{\xi'}x_1$ to $a_1$, $t_2$ for the piece of
$\gamma_2$ from $a_1$ to $a_2$ and $s_2$ for the piece of $\gamma_2$
from $a_2$ to $nx_0 + 4n^{\xi'}x_1$. Similarly, write $s_1'$ for the
piece of $\gamma_1$ from 0 to $a_1'$ and $s_2'$ for the piece of
$\gamma_1$ from $a_2'$ to $nx_0$. By definition of $T_1(x_0;n)''$, we
have the following almost surely:
%
%
\begin{eqnarray}\label{eqtupperbound}
\qquad T_1(x_0;n)'' &\leq&
\tau(0,a_1) + \tau(t_2) + \tau(a_2,ne_1)
\nonumber\\[-8pt]\\[-8pt]
&=& \tau(0,a_1) - \tau(s_1) + \tau(a_2,ne_1)
- \tau(s_2) + T_2\bigl(x_0;n,\xi'\bigr)''.\nonumber
\end{eqnarray}

Set $H_3 = nx_0 + H_0$, and let $\mathcal{C}$ be the set of all points
in $\mathbb{R}^d$ that are $\ell_\infty$ distance at most $5n^{\xi '}$
from the line segment connecting 0 to $nx_0$. Last, let $V_i = H_i
\cap\mathcal{C}$ for $i=0,\ldots, 3$ and
\[
X_i\bigl(n,\xi',\beta\bigr) = \mathop{\max
_{v_1,v_2 \in V_{2i}}}_{w_1,w_2 \in
V_{2i+1}} \bigl|\tau(v_1,w_1)-
\tau(v_2,w_2)\bigr|,\qquad i=0,1.
\]
Using this notation and (\ref{eqtupperbound}), we can give an upper
bound for $T_1(x_0;n)''$ of
\[
T_1(x_0;n)'' \leq T_2\bigl(x_0;n,\xi'\bigr)''+ X_0\bigl(n,\xi',\beta\bigr) + X_1\bigl(n,\xi ',\beta\bigr).
\]
To bound $T_2(x_0;n,\xi')''$, we can similarly write
\begin{eqnarray*}
T_2\bigl(x_0;n,\xi'\bigr)'' &\leq& \tau\bigl(4n^{\xi'}x_1,a_1'
\bigr) - \tau\bigl(s_1'\bigr)
\\
&&{} + \tau \bigl(a_2',nx_0+4n^{\xi'}x_1
\bigr) - \tau\bigl(s_2'\bigr) + T_1(x_0;n)'.
\end{eqnarray*}
Therefore, $T_2(x_0;n,\xi')'' \leq X_0(n,\xi',\beta) + X_1(n,\xi
',\beta) + T_1(x_0;n)''$. Putting these together,
\[
\bigl|\delta T\bigl(x_0;n,\xi'\bigr)''\bigr|
\leq X_0\bigl(n,\xi',\beta\bigr) + X_1
\bigl(n,\xi',\beta \bigr)\qquad\mbox{almost surely}
\]
and consequently
\[
\operatorname{Var}\delta T\bigl(x_0;n,\xi'
\bigr)'' \leq\bigl\|\delta T\bigl(x_0;n,\xi
'\bigr)''\bigr\|_2^2
\leq 2\bigl(\bigl\|X_0\bigl(n,\xi',\beta\bigr)
\bigr\|_2^2 + \bigl\|X_1\bigl(n,\xi',\beta
\bigr)\bigr\|_2^2\bigr).
\]
The variables $X_0$ and $X_1$ are identically distributed, so
$\operatorname{Var} \delta T(x_0;n,\xi')''
\leq4\|X_0(n,\xi',\beta)\|_2^2$. Finally, we combine with
(\ref{eqdecompose}) to get
%
%
\begin{equation}
\label{eqfinal23} \operatorname{Var}\delta T\bigl(n,\xi'\bigr)
\leq4 \bigl\|X_0\bigl(n,\xi',\beta\bigr)\bigr\|
_2^2 +C_2.
\end{equation}

The last step is to invoke Proposition~\ref{propsourav}. The variable
$X_0(n,\xi',\beta)$ is the same as $X(x_0;5n^{\xi'},n^\beta )$ there.
Because $\beta$ was chosen to be larger than $\xi'$, the condition
$5n^{\xi'} \leq\frac{\varepsilon_{x_0}}{2\sqrt {d-1}}n^{\beta}$ holds
for large $n$. Thus there exists $C_3$ such that for all large $n$,
\[
\bigl\|X_0\bigl(n,\xi',\beta\bigr)\bigr\|_2^2
\leq C_3 n^{2\beta(1-\kappa)}n^{2\xi'
\kappa} + C_3n^{2\beta\chi'},
\]
where $\kappa$ is from the statement of this proposition. With
(\ref{eqfinal23}), this completes the proof.
\end{pf}

\subsection{\texorpdfstring{Proof of Theorem \protect\ref{theorem1}}
{Proof of Theorem 2.1}}

We now prove Theorem \ref{theorem1}. Assume that $\chi_a=\chi_b=\chi
> 0$. Further, we may assume $\xi_a<1$ because if $\xi_a=1$, the
relation holds by the bound $\chi\leq1/2$; see Remark~\ref{remtrivial}.

Choose $|x_0|=1$ such that (\ref{eqcurvatureassumption}) holds for some
$\varepsilon_{x_0}$ and $C_{x_0}>0$ for $\kappa=2$. (The existence of
such a point is proved in \cite{Sourav}, Proposition~5.1.) From the
previous two sections, for each choice of $\chi ',\chi'',\xi'$ and
$\beta$ satisfying
%
%
\begin{equation}
\label{eqequation1} 0<\chi''< \chi<
\chi'\quad\mbox{and}\quad \xi_a<\xi'<\beta<1,
\end{equation}
there exist constants $C_i= C_i(\chi',\chi'',\xi',\beta)$ ($i=1,2$)
such that for all $n$,
\[
C_1n^{2\chi''} \leq\operatorname{Var}\delta T\bigl(n,
\xi'\bigr) \leq C_2n^{-2\beta+ 4\xi'} +C_2n^{2\beta\chi'}.
\]
For any $\beta$ with $\xi'<\beta<1$, we may choose $\chi'' = \chi
''(\beta)$ and $\chi'= \chi'(\beta)$ that satisfy~(\ref{eqequation1})
and are so close to $\chi$ that $2\beta\chi' < 2\chi ''$. For such a
choice of $\chi''$ and $\chi'$ we then have
\[
(1/2)C_1n^{2\chi''} \leq C_2n^{-2\beta+ 4\xi'}\qquad\mbox{for all large } n
\]
and, therefore, $\chi'' \leq-\beta+ 2\xi'$. Taking $\beta\uparrow1$ and
noting that $\chi''(\beta) \uparrow\chi$, we find $\chi\leq-1 + 2\xi'$.
This is true for all $\xi' > \xi_a$, so $\chi\leq-1 + 2\xi_a$.

\section*{Acknowledgments}
We are very grateful to S. Chatterjee for his advice and comments on
the presentation in an earlier version of the paper. Also we thank
G.~Ben Arous and C. Newman for suggestions on the \hyperref[sec1]{Introduction} and A.
Sapozhnikov for comments and careful readings of the paper.




\printaddresses

\end{document}